\documentclass{amsart}
\usepackage{latexsym,times,amsthm,amsmath,graphicx,tikz,color}
\usetikzlibrary{positioning}

\usepackage[utf8]{inputenc}
\usepackage[english]{babel}
\graphicspath{{thesis}}

\DeclareMathOperator{\Gl}{GL}
\newcommand{\7}{B^2A^2B}
\newcommand{\0}{A}
\newcommand{\2}{A^2}
\newcommand{\1}{B}
\newcommand{\5}{B^2AB}
\newcommand{\4}{A^2BA}
\newcommand{\3}{B^2}
\newcommand{\calR}{\mathcal{R}}
\newcommand{\calC}{\mathcal{C}}

\renewcommand{\phi}{\varphi}
\newtheorem{theorem}{Theorem}[section]
\newtheorem{lemma}[theorem]{Lemma}
\newtheorem{remark}[theorem]{Remark}
\newtheorem{conjecture}[theorem]{Conjecture}
\newtheorem{proposition}[theorem]{Proposition}
\newtheorem{definition}[theorem]{Definition}
\newtheorem*{recap}{Theorem~\ref{transthm}}

\newtheorem{question}[theorem]{Question}

\title{Hurwitz Transitivity of Longer Reflection Factorizations in G4 and G5}
\author{Zachery Peterson}
\date{}

\begin{document}
	\maketitle
\begin{abstract}
	We prove that the Hurwitz action on reflection factorizations of Coxeter elements is transitive up to certain natural constraints in the complex reflection groups G4 and G5.  This affirms a more general conjecture by Lewis and Reiner in these specific cases.  The proof uses induction on length of the factorization using the fact that the square of a reflection is also a reflection.
\end{abstract}
\section{Introduction}
  The goal of this paper is to provide a proof of the following conjecture by Lewis and Reiner in the specific case of the complex reflection groups G4 and G5.
  \begin{conjecture}[{\cite[Conj. 6.3]{LR}}]\label{conj}
  	In a well-generated finite complex reflection group, two reflection factorizations of a Coxeter element lie in the same Hurwitz orbit if and only if they share the same multiset of conjugacy classes.
  \end{conjecture}
   (Definitions for terminology used here may be found in Section \ref{background}.)  Lewis and Reiner prove a special case of Conjecture \ref{conj} in the case of the real reflection groups.  Their proof involves finding pairs of consecutive equal reflections in reflection factorizations and using them to relate longer factorizations to shorter ones for the purpose of induction \cite{LR}.  This induction is founded by a base case given by a theorem of Bessis \cite{Bess03}, similar to Theorem \ref{thm:Bessis} for real reflection groups.  Our proof here follows a similar strategy with a more elegant form of reduction that takes advantage of some unique facts which are true in reflection groups which have exclusively reflections of order 3. The result is a proof of the following theorem.
	\begin{theorem}\label{transthm}
		Let $W$ be one of the reflection groups G4 or G5 and let $c\in W$ be a Coxeter element. Two reflection factorizations $(t_1,t_2,\ldots,t_k)$ and $(t_1',t_2',\ldots, t_k')$ of $c$ are in the same Hurwitz orbit if and only if they have the same multiset of conjugacy classes.
	\end{theorem}
	Section \ref{background} gives the background information needed to understand the statement of the theorem and the concepts used in the proof.  Section \ref{outline} outlines the strategy used in Section \ref{proofsec} to build up the proof.  Finally, Section \ref{P.S.} gives some additional remarks, discoveries, and open questions that resulted from the exploration of this proof, including possible extensions to other groups and conjectures about counting.
	\section{General Background}\label{background}
		Consider an $n$ dimensional vector space.  This vector space contains a set of $(n-1)$-dimensional \textit{hyperplanes}.  We may define a set of linear transformations in this space called \textit{reflections}, which are the elements of $\Gl(V)$ which fix these hyperplanes.  Suppose this vector space lies over the complex field.  If these reflections generate a finite group, this group is called a \textit{complex reflection group}.
		
		There are 37 irreducible finite complex reflection groups, including three infinite families of groups, as classified by Shephard and Todd \cite{ST}.  The best-known example is the symmetric group $S_m$, the group of permutations of a set of $m$ objects.  It is well-known that $S_m$ is generated by a set of $\binom{m}{2}$ transpositions, which in this case act as the reflections.
		
		Since reflection groups are generated by reflections, every element may be written as a product of reflections.  These products, written as tuples, are called \textit{reflection factorizations}.  One of the interesting ways to manipulate these factorizations is through \textit{Hurwitz moves}.  Given a factorization $(t_1,t_2,\ldots,t_k)$ of an element, the $i$-th Hurwitz move, $\sigma_i$, where $i$ is an integer $1\leq i\leq k-1$, takes it to
		$$
		(t_1,\,t_2,\ldots t_{i-1},\quad t_{i+1},\quad t_{i+1}^{-1}\cdot t_i\cdot t_{i+1},\quad t_{i+2},\ldots,\,t_k).
		$$
		That is to say the $i$-th Hurwitz move transposes the $i$-th and $(i+1)$-th elements of a factorization while conjugating the $i$-th by the $(i+1)$-th.\footnote{It can be easily proven that these moves give rise to an action of the braid group on $k$ strands, and therefore that they are invertible.  This is discussed in greater detail by Bessis \cite{Bess03}.}
		
		It should be noted that this action preserves the product and thus the result is still a reflection factorization of the same element.  The \textit{Hurwitz orbit} of a factorization is the set of all factorizations that can be achieved by performing a series of Hurwitz moves.  Additionally, the Hurwitz move does not affect the conjugacy classes of the pair of elements on which it acts, shifting the position of one and conjugating the other by the first.  Thus, we have the following remark.
			\begin{remark}\label{multisetrem}
				A Hurwitz move on a reflection factorization has no effect on the multiset of conjugacy classes of the elements.
			\end{remark}
		This remark is discussed further by Bessis \cite[Prop. 1.6.1]{Bess03}.	We also rely on the following principle in many of the proofs to come.
		\begin{proposition}\label{sortingprop}
			A reflection factorization with a given multiset of conjugacy classes has, in its Hurwitz orbit, factorizations with all possible permutations of those conjugacy classes.
		\end{proposition}
		This fact is true for all reflection groups.  We will use it to begin proofs with factorizations that have convenient permutations of conjugacy classes for our purposes.
		\begin{proof}
			It suffices to prove that any particular permutation of conjugacy classes may be reached.  Take a permutation of a multiset of conjugacy classes of a reflection factorization.  Scan the reflection factorization from left to right until an element is found out of position.  From here, continue to scan until an element from the class which is desired for that position is found.  Note the position of the former, call it $i$, and the latter, call it $j$.  Perform the braid action $\sigma_j\circ\sigma_{j-1}\circ\cdots\circ\sigma_{i}$.  This has the effect of moving latter element to position $i$ in the factorization without affecting the conjugacy class of any other position prior to $i$.  Iterate this process until all elements are in position.
		\end{proof}
		Recall that reflections are linear transformations.  Note that all linear transformations can be represented as matrices.  Thus, it is permissible to think of complex reflection groups as matrix groups generated by the reflection matrices, which are the irreducible matrix representations of the linear transformations.  Thus, we can define the \textit{rank} of the group in the traditional sense as the rank of the irreducible matrix representations.  With the concept of rank, we introduce the idea of a group being \textit{well-generated}.  We say a complex reflection group $W$ of rank $n$ is \textit{well-generated} if there exists a set of reflections $S$ such that $S$ generates $W$ and $|S|=n$.  Given a well-generated complex reflection group of rank $n$, define $\calR$ to be the set of reflections and $\mathcal{A}$ to be the set of reflecting hyperplanes.  Such a group will have a \textit{Coxeter number} $h$ defined as
		$$
		h=\frac{|\calR|+|\mathcal{A}|}{n}.
		$$
		Elements that have some primitive $h$-th root of unity as an eigenvalue are called \textit{Coxeter elements}.\footnote{There is some division in the mathematical community as to the accepted definition of Coxeter elements.  The alternative, more narrow definition is that Coxeter elements have an eigenvalue of $e^{2\pi i/h}$.  Reiner, Ripoll, and Stump show that many assertions have a bijectional relationship between these two definitions, including all assertions presented here \cite{RRS}.}  These elements are of particular interest because of their properties in complex reflection groups and their Hurwitz orbits.  An example of this was observed by Bessis in the following theorem.
		\begin{theorem}[{Bessis' Theorem \cite[Prop. 7.6]{Bess15}}]
    \label{thm:Bessis}
			Let $W$ be a well-generated complex reflection group and let $c$ be a Coxeter element in $W$. The Hurwitz action is transitive of the set of shortest reflection factorizations of $c$.
		\end{theorem}
		Of note for us, Bessis' Theorem proves Conjecture \ref{conj} for the shortest reflection factorizations.
		
\subsection{G4 Background}		
		One of the simplest complex reflection groups is G4, so called because it is the fourth in the Shephard-Todd list of complex reflection groups \cite{ST}.  The group G4 is generated by a pair of reflections $A$ and $B$ each of which cubes to the identity matrix and have the property that $ABA=BAB$.  It has the property that if $t_i$ is a reflection then its square is also a reflection.  Additionally, due to the nature of the reflection groups which have exclusively reflections of order 3, this element is also the inverse.  One example of a generating pair of reflection matrices is the following
		$$
		A=\begin{bmatrix}
		1 & 0\\
		0 & \omega
		\end{bmatrix},\quad
		B=\begin{bmatrix}
		\frac{\omega-\omega^2}{3} & \omega^2\\
		\\\frac{-2\omega^2}{3} & \frac{-\omega-2\omega^2}{3}
		\end{bmatrix}
		$$
		where $\omega=e^{\frac{2\pi i}{3}}$ is one of the complex third roots of unity.
		
		G4 contains 24 elements, each of which is a $2\times2$ complex matrix.  It contains eight Coxeter elements and eight reflections, both split into two conjugacy classes.  The set of all reflections may be written as
		\begin{equation*}
		\calR=\{\underbrace{A,\,B,\,A^2BA,\,B^2AB}_{\calR_1},\underbrace{A^2,\,B^2,\,A^2B^2A,\,B^2A^2B}_{\calR_2}\}
		\end{equation*}
		and the set of all Coxeter elements $\mathcal{C}$ may be written as
		\begin{equation*}
		\mathcal{C}=\{\underbrace{A^2B^2,\,B^2A^2,\,BA^2B,\,AB^2A}_{\calC_1},\underbrace{AB,\,BA,\,B^2AB^2,\,A^2BA^2}_{\calC_2}\}.
		\end{equation*}
		In the reflections, the first four elements form one conjugacy class, call it $\calR_1$, and the last four form the other, call it $\calR_2$.  The eigenvalues of the reflections are 1 and $\omega$ for $\calR_1$ and 1 and $\omega^2$ for $\calR_2$.  The Coxeter elements have eigenvalues $-1$ and $-\omega$ for the first four, call them $\calC_1$, and $-1$ and $-\omega^2$ for the last four, $\calC_2$.  Note that $-\omega=e^{\frac{5\pi i}{3}}$, which makes $-\omega$ and $-\omega^2$ primitive sixth roots of unity.
		
		\begin{figure}[h]
			\centering
			\includegraphics[scale=.1]{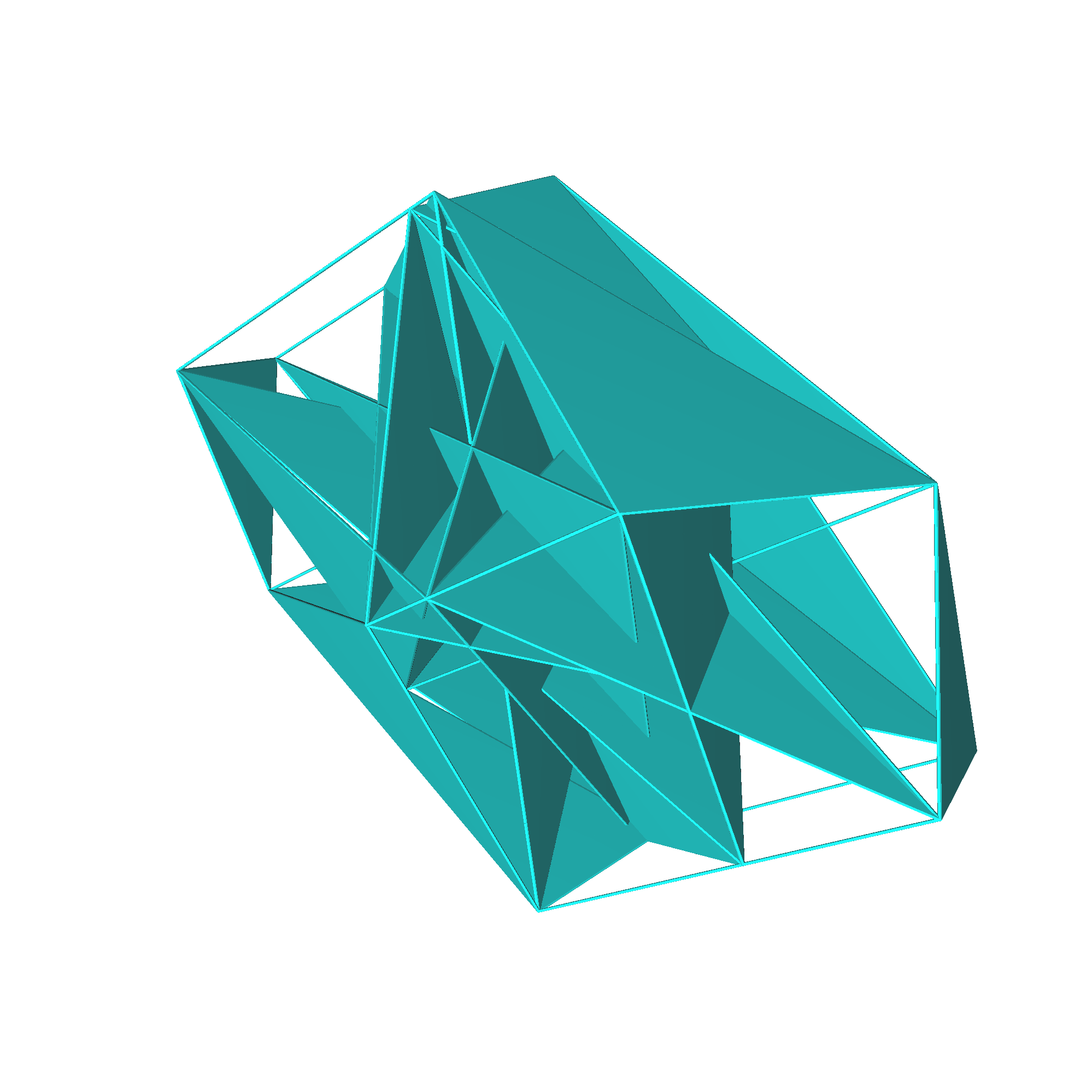}
			\caption{A computer generated three dimensional image of G4.  The 24 vertices each represent one of the 24 elements of the group.  Two vertices are connected if one may be expressed as a product of the other and some reflection \cite{Liz}.}\label{G4im}
		\end{figure}
		
		Some of the information presented in the following sections is the result of empirical analysis of data.  In this analysis, we studied a single Coxeter element, rather than all of them.  We are able to do this because of the results of Reiner, Ripoll, and Stump \cite{RRS}.  Although our choice of Coxeter element was arbitrary, in examples we refer mainly to the Coxeter element
		$$
		AB=\begin{bmatrix}
		\frac{\omega-\omega^2}{3}&\omega^2\\\\
		\frac{-2}{3} & \frac{2\omega+\omega^2}{3}
		\end{bmatrix}.
		$$
		The Coxeter element $AB$ has eigenvalues $-1$ and $-\omega^2$ with eigenvectors $(\frac{\omega+2}{2},1)^T$ and $(\omega^2-1,1)^T$ respectively.  It has shortest factorizations $(A,B),\,(B,B^2AB)$, and $(B^2AB,A)$, and it is a member of $\calC_2$.

\subsection{G5 Background}
	Like G4, G5 is a relatively small complex reflection group of rank two generated by a pair of reflections $A$ and $B$.  These reflections, along with all other reflections in the group, have order 3, meaning they also cube to the identity and have the property that their square is their inverse.  The major structural differences in G5 stem from a defining property to do with its generators.  Where G4 has the relation $ABA=BAB$, G5 has the relation $ABAB=BABA$. 
	
	We outline a few specific facts about G5 which make it different from G4. To begin with, G5 has 16 reflections split into four conjugacy classes of four reflections each which we will call $\calR_1,\calR_2,\calR_3$, and $\calR_4$. These four conjugacy classes have the following properties.
	\begin{enumerate}
		\item For every reflection in $\calR_1$, its square and inverse is in $\calR_3$.  The same is true of $\calR_2$ and $\calR_4$.  We will refer to these as \textit{square classes}.
		\item For every reflection in $\calR_1$, there is an reflection in $\calR_2$ with which the former commutes and has the same eigenvalue.  The same is true of $\calR_3$ and $\calR_4$.  We call these pairs of reflections \textit{adjacent pairs} denoted with a $~\widehat{~}~$ and call these conjugacy classes \textit{adjacent classes}.
		\item For every reflection in $\calR_1$ there is a reflection in $\calR_4$ whose square is the adjacent pair of that reflection.  These pairs of elements are called \textit{semi-squares} denoted $\widehat{~}~^2$.  The same is true for $\calR_2$ and $\calR_3$.  We call these classes \textit{non-adjacent}.
		\item For every reflection, there are is a reflection in each other conjugacy class with which it commutes.  These four-tuples of commuting reflections we call \textit{commuting sets}.  Commuting sets consist of an element, its square, its adjacent pair, and its semi-square.
	\end{enumerate}
	\begin{figure}[h]
		\begin{tikzpicture}[n/.style={rectangle,draw=black,minimum size=1mm},i/.style={rectangle,draw=white, minimum size=1mm}]
		\node[n](center){$\calR_1$};
		\node[i](uplabel)[right=of center, label={$\widehat{t}$}]{};
		\node[n](top)[right=of uplabel]{$\calR_2$};
		\node[i](llabel)[below=of center]{$t^2\quad\quad$};
		\node[n](left)[below=of llabel]{$\calR_3$};
		\node[i](rlabel)[below right=of center]{$\quad\widehat{t}~^2$};
		\node[i](rrlabel)[below=of top]{$\quad\quad t^2$};		
		\node[n](right)[below=of rrlabel]{$\calR_4$};
		\node[i](blabel)[right= of left, label={$\widehat{t}$}]{};
		
		\draw[ultra thick, green] (center)--(right);
		\draw[ultra thick,blue] (top)--(right);
		\draw[ultra thick,red] (top)--(center);
		\draw[ultra thick,red] (right)--(left);
		\draw[ultra thick,blue](center)--(left);
		\draw[ultra thick,green](left)--(top);	
		\end{tikzpicture}
		\caption{The relationships between the four conjugacy classes of G5 form a complete graph on four vertices.  Each edge represents a relationship between classes--either square classes (with a $t^2$) if the elements are squares, adjacent classes (with a $~\widehat{t}~$) if they have commuting elements of the same eigenvalues, or non-adjacent classes (with a $~\widehat{t}~^2$) if they have elements whose squares are adjacent elements.  We have color-coded the relationships--blue for $~^2$, red for $~\widehat{~}~$, and green for $~\widehat{~}~^2$.}\label{G5graph}
	\end{figure}
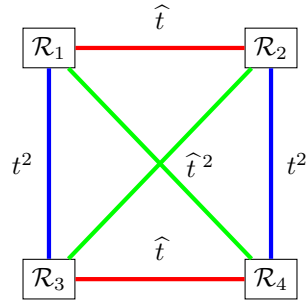
	The organization of these conjugacy classes gives rise to an intriguing remark.  We will make use of this remark to simplify certain cases to come.
	\begin{remark}\label{g5sub}
		G5 contains multiple copies of G4 as subgroups.  In particular, one of these subgroups has, as its reflections, $\calR_1\cup\calR_3$ while the other has $\calR_2\cup\calR_4$.
	\end{remark}
	It is not immediately obvious why this fact is true considering the different generating structures of the two groups, but it may be observed through empirical analysis.  As an analogue of this phenomenon, consider the more well-known family of dihedral groups, in particular the symmetries of the hexagon and triangle.  In the case of the hexagon, the reflections fall into two conjugacy classes which are inverses of each other.  A pair of non-square reflections from different conjugacy classes will generate the whole group, however a pair of reflections from within conjugacy classes will generate a subgroup which is isomorphic to the symmetries of the triangle.

\section{Proof Outline}\label{outline}
	Complex reflection groups with reflections of exclusively order three have the property that if $t$ is a reflection, so is $t^2$.  This fact leads to a connection between reflection factorizations of length $m$ with a consecutive $(t,t)$ pair and reflection factorizations of length $(m-1)$ where this pair is replaced by $t^2$.  The strategy for this paper is to exploit this fact in a useful manner.
	
	We must first demonstrate that given a reflection factorization, we can find reflection factorization in its Hurwitz orbit where one of these consecutive $(t,t)$ pairs appears.  To that end, \textit{Step 1} gives the details regarding the orbits of pairs of reflections from all combinations of conjugacy classes, namely their size and what elements appear therein.  These facts are utilized in \textit{Step 2} in order to show that every reflection factorization of a given length has in its Hurwitz orbit a reflection factorization with a consecutive $(t,t)$ pair.  In these two steps, the differences between G4 and G5 are multiple and make the sections more disparate.
	
	The next phase is to make concrete the relationship between the reflection factorizations with the consecutive $(t,t)$ and their shorter counterparts.  This sets up an inductive argument.  Base cases are outlined briefly in \textit{Step 3}.  Finally, \textit{Step 4} defines several structures which provide a mapping from Hurwitz moves made in shortened factorizations and the original longer factorizations.  In these steps, there are far fewer differences between the two groups, so they are more condensed.  The result is that reflections lie in the same Hurwitz orbit if there shortened counterparts do.  We then complete the proof.
	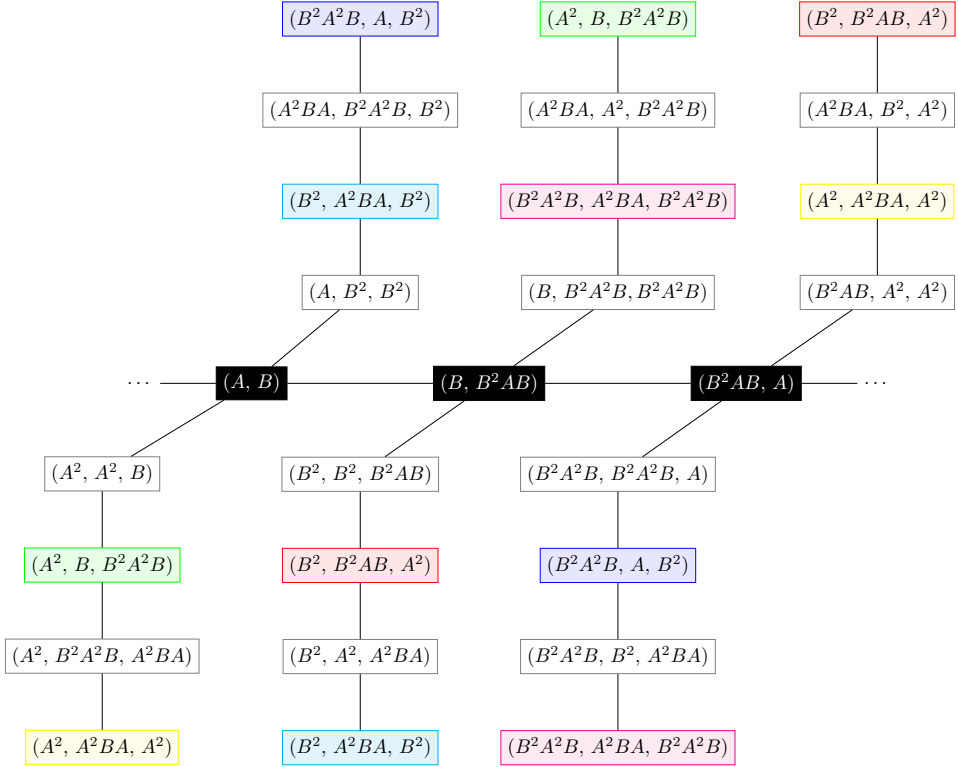
\begin{figure}[h]
		\resizebox{5in}{4in}{
			\begin{tikzpicture}[
			rednode/.style={rectangle, draw=red,fill=red!10,minimum size=1mm},
			bluenode/.style={rectangle,draw=blue,fill=blue!10,minimum size=1mm},
			greennode/.style={rectangle,draw=green,fill=green!10,minimum size=1mm},
			greynode/.style={rectangle,draw=gray,minimum size=1mm},
			wnode/.style={rectangle,fill=white,minimum size=6mm},
			ynode/.style={rectangle,draw=yellow,fill=yellow!10,minimum size=1mm},
			cnode/.style={rectangle,draw=cyan,fill=cyan!10,minimum size=1mm},
			mnode/.style={rectangle,draw=magenta,fill=magenta!10,minimum size=1mm},
			blacknode/.style={rectangle,draw=black,fill=black,minimum size=1mm}
			]
			\node[blacknode](righ){\textcolor{white}{$(\5,\,\0)$}};
			\node[wnode](rblank)[left=of righ]{};
			\node[blacknode](cent)[left=of rblank]{\textcolor{white}{$(\1,\,\5)$}};
			\node[wnode](lblank)[left=of cent]{};
			\node[blacknode](lef)[left=of lblank]{\textcolor{white}{$(\0,\,\1)$}};
			\node[wnode](llblank)[left=of lef]{$\cdots$};
			\node[wnode](ror)[right=of righ]{$\cdots$};
			\node[greynode](1)[above=of lblank]{$(\0,\,\3,\,\3)$};
			\node[greynode](2)[above=of rblank]{$(\1,\,\7,\7)$};
			\node[greynode](3)[above=of ror]{$(\5,\,\2,\,\2)$};
			\node[greynode](4)[below left=of lef]{$(\2,\,\2,\,\1)$};
			\node[greynode](5)[below=of lblank]{$(\3,\,\3,\,\5)$};
			\node[greynode](6)[below=of rblank]{$(\7,\,\7,\,\0)$};
			\node[cnode](7)[above=of 1]{$(\3,\,\4,\,\3)$};
			\node[greynode](8)[above=of 7]{$(\4,\,\7,\,\3)$};
			\node[bluenode](9)[above=of 8]{$(\7,\,\0,\,\3)$};
			\node[mnode](10)[above=of 2]{$(\7,\,\4,\,\7)$};
			\node[greynode](11)[above=of 10]{$(\4,\,\2,\,\7)$};
			\node[greennode](12)[above=of 11]{$(\2,\,\1,\,\7)$};
			\node[ynode](13)[above= of 3]{$(\2,\,\4,\,\2)$};
			\node[greynode](14)[above=of 13]{$(\4,\,\3,\,\2)$};
			\node[rednode](15)[above=of 14]{$(\3,\,\5,\,\2)$};
			\node[greennode](16)[below=of 4]{$(\2,\,\1,\,\7)$};
			\node[greynode](17)[below=of 16]{$(\2,\,\7,\,\4)$};
			\node[ynode](18)[below=of 17]{$(\2,\,\4,\,\2)$};
			\node[rednode](19)[below=of 5]{$(\3,\,\5,\,\2)$};
			\node[greynode](20)[below=of 19]{$(\3,\,\2,\,\4)$};
			\node[cnode](21)[below=of 20]{$(\3,\,\4,\,\3)$};
			\node[bluenode](22)[below=of 6]{$(\7,\,\0,\,\3)$};
			\node[greynode](23)[below=of 22]{$(\7,\,\3,\,\4)$};
			\node[mnode](24)[below=of 23]{$(\7,\,\4,\,\7)$};
			
			\draw (lef)--(4);
			\draw (lef)--(1);
			\draw (lef)--(cent);
			\draw (cent)--(2);
			\draw (cent)--(5);
			\draw (cent)--(righ);
			\draw (righ)--(3);
			\draw (righ)--(6);
			\draw (1)--(7);
			\draw (7)--(8);
			\draw (8)--(9);
			\draw (2)--(10);
			\draw (10)--(11);
			\draw (11)--(12);
			\draw (3)--(13);
			\draw (13)--(14);
			\draw (14)--(15);
			\draw (4)--(16);
			\draw (17)--(18);
			\draw (16)--(17);
			\draw (5)--(19);
			\draw (19)--(20);
			\draw (20)--(21);
			\draw (6)--(22);
			\draw (22)--(23);
			\draw (23)--(24);
			\draw (righ)--(ror);
			\draw (lef)--(llblank);
			
			\end{tikzpicture}}
		\caption{The transitivity of the length-three reflections factorizations of $AB$ depending on the transitivity of the length-two reflection factorizations (pictured in black).  Each line represents a Hurwitz move.  The colored factorizations appear twice.}\label{tree} 
	\end{figure} 
		
	As an illustration of the idea, Figure \ref{tree} shows reductions from length-three to length-two reflection factorizations in G4 as they relates to Hurwitz transitivity.  Pick any two length-three reflection factorizations $T_1,T_2$ of $AB$, note that each colored element appears twice.  Starting from $T_1$, move down the diagram until you reach one of the elements in its orbit with the desired $(t,t)$ pair.  From here, reduce to a length-two factorization, pictured in black.  To get to $T_2$, find which length-two factorization it stems from and move to it.  From there, switch back to a length-three factorization and follow up the branches until you are at $T_2$.

\section{The Proof}\label{proofsec}
	  
\subsection{G4 Step 1}\label{rems}
	Our first task is to establish some facts about the Hurwitz orbits of length-two reflection factorizations.  These next remarks seek to answer the question of which reflections will show up and where in length-two factorizations.  In the following, we say that a reflection $t$ appears in an orbit if the orbit contains the factorization $(\cdot,t)$ where $\cdot$ represents some reflection.  Note that this occurs if and only if $(t,\cdot)$ is also in the orbit, so if a reflection appears, it can occupy either position in the factorization.\footnote{This follows in one direction from a single application of a Hurwitz move, and in the other direction by an application of the inverse Hurwitz move.}
	
	\begin{remark}\label{2facrem}
		A reflection factorization in G4 which consists of exactly one pair of elements from the same conjugacy class which are not equal has an orbit of size three in which exactly three of the four elements of one conjugacy class appear.\footnote{Incidentally, the Coxeter elements all fall into this category.  The conjugacy class of the reflections is the one whose eigenvalues are the squares of those of the factored Coxeter element.  Additionally, the reflection which does not appear in the orbit of the shortest factorization of a Coxeter element $c$ is $c^2$.}
	\end{remark}
	This may be easily observed through empirical analysis of the orbits.  
	
	For example, consider the Coxeter element $AB$.  The most obvious length-two factorization is $(A,B)$ which has orbit
	$$
	\{(A,B),\quad \big(B,B^2AB\big),\quad \big(B^2AB,A\big) \}
	$$
	in which the elements $A$, $B$, and $\5$, three of the four elements of $\calR_1$, all appear.
	
	Remark \ref{2facrem} takes into consideration all pairs of the same conjugacy class with the exception of $(t,t)$, which has an orbit with size one containing only itself.  As far as the pairs from different conjugacy classes, we have $(t,t^2)$, whose orbit is of size two, namely 
	$$
	\{(t,t^2),(t^2,t) \}.
	$$
	For the rest of the pairs from different conjugacy classes, we have the following.
	\begin{remark}\label{neqrem}
		A reflection factorization in G4 which consists of exactly a pair of elements $t_1\in\calR_1,t_2\in\calR_2$ which are not inverses of each other has an orbit of size four.  Four elements appear in this orbit, call them $t_1,\,t_2,\,t_1',$ and $t_2'$, with $t_1'\in\calR_1$, $t_2'\in\calR_2$.  Additionally,
		$$
		t_1^2\neq t_2',\quad t_1'^2\neq t_2,\quad t_1'^2\neq t_2'.
		$$
		Which is to say,
		$$
		\calR_1=\{t_1,t_1',t_2^2,t_2'^2\}
		$$
		and similarly for $\calR_2$.
	\end{remark}
	For example, consider the element $AB^2$.  This has an obvious factorization $(A,B^2)$ which has orbit:
	$$
	\{(A,\,B^2),~(B^2,\,\4\big),~(\4,\,\7),~(\7,\,A) \}.
	$$
	Note that $A,\,\4\in\calR_1$ and $B^2,\,\7\in\calR_2$.  Also note that none of these elements are inverses of each other.  In particular,
	\begin{align*}
		\calR_1&=\{A,\,\4\}\cup\{(B^2)^2,\,(\7)^2\}\\
		&=\{A,\,\4\}\cup\{B,\,\5\}
	\end{align*}
	and similarly,
	\begin{align*}
		\calR_2&=\{B^2,\,\7 \}\cup\{A^2,\,(\4)^2 \}\\
		&=\{B^2,\,\7 \}\cup\{A^2,\,A^2B^2A \}.
	\end{align*}
	
\subsection{G5 Step 1} 
	To begin with, Remark \ref{2facrem} holds in G5 since it operates within the G4 subgroups described in Remark \ref{g5sub}, as does the statement that the orbit of $(t,t)$ has size one, which is true in any group.  The rest of the remarks from Section \ref{rems} require some additional interpretation due to their being a more complex conjugacy class structure.  These remarks are drawn from empirical analysis of the group made using Sage software \cite{sage}.
	\begin{remark}\label{g5commrem}
		A reflection factorization in G5 which contains exactly a pair of non-equal elements from within a commuting set has a Hurwitz orbit of size two.
	\end{remark}
	This remark is true of any pair of elements which commute with each other in any group.
	\begin{remark}\label{g5squarerem}
		A reflection factorization which contains exactly a pair of non-commuting elements from square classes has a Hurwitz orbit of size four.  Specifically, suppose $t_1\in\calR_1$ and $t_3\in\calR_3$ with $t_1\neq t_3^2$.  Take $t_2\in\calR_1$ and $t_4\in\calR_3$ to appear in the orbit of $(t_1,t_3)$.  Then
		$$
		\calR_1=\{t_1,t_2,t_3^2,t_4^2\}.
		$$
	\end{remark}
	\begin{remark}\label{g5permrem}
		A reflection factorization in G5 which contains exactly a pair of non-commuting elements from adjacent classes has a Hurwitz orbit of size four.  Specifically, suppose $t_1\in\calR_1$ and $t_3\in\calR_2$ with $t_1\neq\widehat{t_3}$.  Take $t_2\in\calR_1$ and $t_4\in\calR_2$ to appear in the orbit of $(t_1,t_3)$.  Then
		$$
		\calR_1=\{t_1,t_2,\widehat{t_3},\widehat{t_4}\}.
		$$
	\end{remark}
	Incidentally, the shortest reflection factorizations of Coxeter elements are among these pairs.
	\begin{remark}\label{g5nonadjrem}
		A reflection factorization in G5 which contains exactly a pair of non-commuting elements from non-adjacent classes has a Hurwitz orbit of size six.
	\end{remark}

\subsection{G4 Step 2}\label{pairs}
	 This section is devoted to proving the following lemma.
	 \begin{lemma}\label{pairlem}
	 	Every reflection factorization with length at least three of a Coxeter element in G4 has, in its Hurwitz orbit, a reflection factorization which has a consecutive $(t,t)$ pair.
	 \end{lemma}
	Before attempting to prove this result, we first make a few observations about some of the shorter reflection factorizations of Coxeter elements.
	\begin{proposition}\label{3conrem}
		Length-three reflection factorizations of Coxeter elements in G4 contain elements from both conjugacy classes.  Reflection factorizations of Coxeter elements with length at least four have at least three elements from a single conjugacy class.
	\end{proposition}
	\begin{proof}
		Recall that the product of the determinants of square matrices is the determinant of the product of those matrices.  Note that the determinants of elements from $\calR_1$ and $\calC_1$ are $\omega$ and those from $\calR_2$ and $\calC_2$ are $\omega^2$.  Observe that for length three, if all elements were from the same class, the product of their determinants would be 1, which is not the determinant of a Coxeter element.  Similarly for length four if there were two of each conjugacy class.  For those with length greater than four this is a clear application of the pigeonhole principle.
	\end{proof}
	We proceed by considering cases.  It is obvious that there are two possibilities for any reflection factorization--either it will contain only elements of one conjugacy class or it will contain elements of both.  We will prove that these two possible cases  will always have, in their Hurwitz orbit, a consecutive $(t,t)$ pair, or \textit{desired pair}.  We start with those which contain elements of both conjugacy classes.  For the sake of simplicity, we will consider length-three reflection factorizations separately.
	\begin{proposition}\label{3facprop}
		Every length-three reflection factorization of a Coxeter element in G4 has in its Hurwitz orbit a factorization with a desired pair.
	\end{proposition}
	\begin{proof}
		Consider a reflection factorization of the form $c=(t_1,t_2,t_3)$.  Recall that Proposition \ref{3conrem} gives that this factorization has elements from both conjugacy classes.  We may, without loss of generality, say that there are two elements from $\calR_1$ and one from $\calR_2$.  By Proposition \ref{sortingprop}, we may take $t_1\in\calR_2$ and $t_2,t_3\in\calR_1$.  Suppose the desired pair does not exist in this factorization.  Observe that if $t_1=t_2^2$, then $c=t_3$ which is to say that the Coxeter element is a reflection, which is not possible.  Thus, $t_1\neq t_2^2$.  Consider the orbit of $(t_1,t_2)$.  By Remark \ref{neqrem}, either $t_3$ or $t_3^2$ appears in this orbit.  By a similar logic, $t_3^2$ cannot possibly appear, so $t_3$ must.  Therefore, the factorization has in its Hurwitz orbit one which contains the desired consecutive $(t,t)$ pair.
	\end{proof}
	\begin{proposition}\label{3facprop2}
		Any reflection factorization with length at least four of a Coxeter element in G4 that contains elements from both conjugacy classes has in its Hurwitz orbit one with a desired pair.
	\end{proposition}
	\begin{proof}
		Consider a reflection factorization of the form $(t_1,t_2,t_3,t_4,\ldots)$.  Recall that Proposition \ref{3conrem} gives that such a reflection factorization will have at least three elements of one conjugacy class.  We may, without loss of generality, take this class to be $\calR_1$.  By Proposition \ref{sortingprop}, we may take $t_1\in\calR_2$ and $t_2,t_3,t_4\in\calR_1$.  Suppose the desired pair does not exists in this factorization.  Suppose $t_1=t_2^2$.  Apply $\sigma_2$.  Since $t_2\neq t_3$ this implies $t_3^2\neq t_1$.  Thus, we may consider only the case $(t_1,t_2,t_3,t_4,\ldots)$ with no $(t,t)$ or $(t,t^2)$ pairs.  Remark \ref{neqrem} gives that two elements of $\calR_1$ will appear in the orbit of $(t_1,t_2)$ and Remark \ref{2facrem} gives that three will appear in the orbit of $(t_3,t_4)$.  By the pigeonhole principle, one of the elements of $\calR_1$ will appear twice, thus producing the desired consecutive $(t,t)$ pair.
	\end{proof}
	Finally, we wish to consider those reflection factorizations which have all elements from the same conjugacy class.  Recall that Proposition \ref{3conrem} tells us that length-three reflection factorizations of Coxeter elements do not fall under this category.  Since we are only concerned with reflection factorizations of Coxeter elements, we need not consider length-three factorizations of this type.
	\begin{proposition}\label{4facprop}
		Any reflection factorization in G4 which contains four elements of one conjugacy class has in its Hurwitz orbit a reflection factorization with a desired pair.
	\end{proposition}
	\begin{proof}
		Consider a reflection factorization of a Coxeter element of the form $(t_1,t_2,t_3,t_4,\ldots)$ with $t_1,t_2,t_3,$ and $t_4$ all in the same conjugacy class.  Suppose the desired pair does not exist in this factorization.  Thus, these four elements are the four elements of a conjugacy class.  In this case, consider the orbit of $(t_2,t_3)$.  By Remark \ref{2facrem}, either $t_1$ or $t_4$ will appear in this orbit, giving us the desired consecutive $(t,t)$ pair.
	\end{proof}
	This result is a more generalized version than we need, but it suffices.  Since this covers all possible cases for longer reflection factorizations of Coxeter elements, we may are ready to prove the lemma.
	
	\begin{proof}[Proof of Lemma \ref{pairlem}]
		Follows from Propositions \ref{3conrem}, \ref{3facprop}, \ref{3facprop2}, and \ref{4facprop}. Proposition \ref{3conrem} presents possible cases for Coxeter elements of certain lengths and the rest of the propositions show that those lengths always contain in their Hurwitz orbits a desired pair.
	\end{proof}
	
\subsection{G5 Step 2}
	The proof of Lemma \ref{pairlem} uses facts which are not generally applicable to all reflection groups, so we must build up a similar lemma from scratch.  
	\begin{lemma}\label{g5pairlem}
		Every reflection factorization with length at least five of a Coxeter element in G5 has, in its Hurwitz orbit, a reflection factorization which has a consecutive $(t,t)$ pair.
	\end{lemma}
	
	We now make observations concerning the Hurwitz orbits of reflection factorizations, particularly those in and around length five.
	\begin{proposition}\label{g5cases}
		Reflection factorizations of length at least five of Coxeter elements in G5 fall under one of the following categories.
		\begin{enumerate}
			\item At least four reflections of one conjugacy class.
			\item At least three reflections of one conjugacy class and two other reflections in any combination of conjugacy classes.
			\item At least two reflection factorizations from a pair of conjugacy class and one from a third.
		\end{enumerate}
	\end{proposition}
	\begin{proof}
		For the length five factorizations, it may be observed via empirical analysis of the Hurwitz orbits that the final possible combination of conjugacy classes, i.e. two from one conjugacy class and one from each of the rest, does not occur for factorizations of Coxeter elements.  For length at least six, all of these are a consequence of the pigeonhole principle.
	\end{proof}
	Proposition \ref{g5cases} lays out the three cases for which we need to prove the presence of a desired pair in the Hurwitz orbit of a general reflection factorization of a Coxeter element.  We present these proofs in order.
	
	First, we observe that Proposition \ref{4facprop} is still true in G5, since it occurs within the G4 subgroups of G5 mentioned in Remark \ref{g5sub}.  
	\begin{proposition}\label{g5threetwo}
		Any reflection factorization in G5 which contains at least three reflections of a single conjugacy class $\calR_i$ and two reflections in any combination of conjugacy classes will have in its Hurwitz orbit a factorization with a desired pair of elements in $\calR_i$.
	\end{proposition}
	\begin{proof}
		We further break this argument into four cases.  In each, we assume without loss of generality that $\calR_i=\calR_1$.  The first has to do with possible combinations of the adjacent and square sets while the second covers all possible combinations with the non-adjacent set.
		\begin{enumerate}
			\item[\textbf{Case 1}:] Suppose that the factorization has multiset of conjugacy classes $\{1^i,2^j,3^k,4^p\}$ with $i\geq3,j+k\geq2,p\geq0$.  Proposition \ref{sortingprop} allows us to choose a factorization $T=(t_1,t_2,t_3,t_4,t_5,\ldots)$ in the Hurwitz orbit such that $t_1,t_2\in\calR_2\cup\calR_3$ and $t_2,t_3,t_4\in\calR_1$.  From here, we must select a factorization $T'$ in the Hurwitz orbit in such a manner that no consecutive elements are in the same commuting set.\footnote{In future cases, we include this step in our selection of $T$.}  Suppose T does not have a repeated element of $\calR_1$.  Therefore the three elements of $\calR_1$ are in different commuting sets.  Suppose $t_1$ is in the same commuting set as $t_2$.  Then it cannot be in the same commuting set as $t_3$, so we may apply a single Hurwitz move to get $(t_1,t_3,t_2',t_4,t_5,\ldots)$ to get a factorization whose first two elements are in different commuting sets.  Suppose $t_5$ is in the same commuting set as $t_4$.  If it is also in the same commuting set as $t_2'$ then $t_2'=t_4$ in which case we're done.  Otherwise, we may perform a single Hurwitz move to get $T':=(t_1,t_3,t_4,t_2'',t_5)$.  The factorization $T'$ now has no relevant consecutive reflections from the same commuting set.  Consider the orbit of $(t_1,t_2)$.  Either by remark \ref{g5permrem} or \ref{g5squarerem}, whichever applies, an additional element of $\calR_1$ will appear in this orbit.  The same is true of the orbit of $(t_4'',t_5)$.  In total, this is five elements of $\calR_1$ appearing in the orbit.  By the pigeonhole principle, one must repeat.
			
			\item[\textbf{Case 2}:]  Suppose that the factorization has a multiset of conjugacy classes $\{1^i,2^j,3^k,4^p\}$ with $i\geq3,j\geq0,k\geq0,p\geq1$.  By Proposition \ref{sortingprop}, we may choose a factorization $T=(t_1,t_2,t_3,t_4,\ldots)$ in the Hurwitz orbit such that $t_1\in\calR_4$ and $t_2,t_3,t_4\in\calR_1$.  Similarly to above, we may do this in a manner such that no consecutive elements are part of the same commuting sets, supposing no elements of $\calR_1$ repeat in $T$.  Consider the orbit of $(t_1,t_2)$.  By Remark \ref{g5nonadjrem}, this orbit will produce two additional elements of $\calR_1$.  In total, this is five elements of $\calR_1$ appearing in the orbit.  By the pigeonhole principle, one must repeat.
		\end{enumerate}
		This covers all possible cases for reflection factorizations of length at least five of Coxeter elements in G5 of the given description, concluding the proof. 
\end{proof}
	\begin{proposition}\label{g5twotwoone}
		Any reflection factorization in G5 which contains at least two reflections from the conjugacy class $\calR_i$, one from its non-adjacent set, and one additional element will have in its Hurwitz orbit a reflection factorization with a desired pair of elements in $\calR_i$.
	\end{proposition}
	\begin{proof}
		We further break this argument into cases.  In each, we assume without loss of generality that $\calR_i=\calR_1$.  There are four cases, one for each of the possible sets for the other element.
		\begin{enumerate}
			\item[\textbf{Case 1}:] We may, without loss of generality, suppose that the factorization has multiset of conjugacy classes $\{1^i,2^j,3^k,4^p\}$ with $i\geq2,j\geq0,k\geq0,p\geq2$.  Proposition \ref{sortingprop} allows us to choose a factorization $T=(t_1,t_2,t_3,t_4,\ldots)$ in the Hurwitz orbit such that $t_1,t_4\in\calR_4$ and $t_2,t_3\in\calR_1$.  From here, we must select a factorization $T'$ in the Hurwitz orbit in such a manner that no consecutive elements are in the same commuting set.  Suppose no elements of $\calR_1$ repeat.  Therefore the elements of $\calR_1$ must be in different commuting sets.  Suppose $t_1$ is in the same commuting set as $t_2$.  Then it is not in the same commuting set as $t_3$, so we apply a single Hurwitz move to get $(t_1,t_3,t_2',t_4,\ldots)$.  Suppose that $t_4$ is in the same commuting set as $t_2'$.  By Remark \ref{2facrem}, applying a single Hurwitz move gives $T':=(t_1,t_2',t_2,t_4,\ldots)$.  Since $t_2\neq t_2'$ and $t_2\neq t_3$ it must be that $T'$ has no relevant adjacent reflections in the same commuting sets.  Consider the orbit of $(t_1,t_2')$.  By Remark \ref{g5nonadjrem}, two additional elements of $\calR_1$ will appear in this orbit.  The same is true of the orbit of $(t_2,t_4)$.  In total, this is six elements of $\calR_1$ appearing in the orbit.  By the pigeonhole principle, one must repeat.			
			\item[\textbf{Case 2}:] Suppose that the factorization has multiset of conjugacy classes $\{1^i,2^j,3^k,4^p\}$ with $i\geq2,j+k\geq1,p\geq1$.  By Proposition \ref{sortingprop},  we may choose a factorization $T=(t_1,t_2,t_3,t_4,\ldots)$ in the Hurwitz orbit such that $t_1\in\calR_4$, $t_2,t_3\in\calR_1$, and $t_4\in\calR_2\cup\calR_3$.  Similarly to above, we may do this in such a manner that no consecutive elements are in the same commuting set.  Suppose that no elements of $\calR_1$ repeat.  Consider the orbit of $(t_1,t_2)$.  By Remark \ref{g5nonadjrem}, two additional elements of $\calR_1$ will appear in this orbit.  Now consider the orbit of $(t_3,t_4)$.  By Remark \ref{g5permrem} or \ref{g5squarerem}, whichever applies, an additional element of $\calR_1$ will appear in this orbit.  In total, this is five elements of $\calR_1$ appearing in the orbit.  By the pigeonhole principle, one must repeat.
			\item[\textbf{Case 3}:] Suppose that the factorization has multiset of conjugacy classes $\{1^i,2^j,3^k,4^p\}$ with $i\geq3,j\geq0,k\geq0,p\geq1$.  The proof is identical to that of Case 2 of the proof of Proposition \ref{g5threetwo}.
		\end{enumerate}
		This covers all possible cases for reflection factorizations of length at least five of Coxeter elements in G5 of the given description, concluding the proof.   
	\end{proof}
	With these propositions completed, we are ready to prove the lemma.
	\begin{proof}[Proof of Lemma \ref{g5pairlem}]
		Proposition \ref{g5cases} lays out the three possible cases for the length-five reflection factorizations of Coxeter elements. We observe that Proposition \ref{4facprop} holds in G5 since it operates within a G4 subgroup mentioned in Remark \ref{g5sub}, which proves Case 1.  Case 2 is proven by Proposition \ref{g5threetwo}.  Proposition \ref{g5twotwoone} is a refinement of Case 3 since the case of two reflections each from a pair of conjugacy classes falls within the case of two reflections from one class and a third from another.  Since this covers all possible cases, this concludes the proof. 
	\end{proof}

\subsection{G4 Step 3}
	Lemma \ref{pairlem} applies to reflection factorizations of length at least three of Coxeter elements.  Since G4 has rank two, the smallest possible reflection factorizations of Coxeter elements have length two, so those are the only base case to account for.  Recall Bessis' Theorem \ref{thm:Bessis}  proves Theorem \ref{transthm} for length-two reflection factorizations of rank two groups,so it may be used as a base case for length two in both G4 and G5.	
\subsection{G5 Step 3}
	Since Lemma \ref{g5pairlem} only applies to those reflection factorizations of length at least five of Coxeter elements, we must provide a proof of additional bases cases, namely length-three and length-four reflection factorizations of Coxeter elements.  To accomplish this, we must rely on some empirical data.
	\begin{proposition}\label{g5induction}
		Theorem \ref{transthm} holds in G5 for reflection factorizations of length less than or equal to four of Coxeter elements.
	\end{proposition}
	\begin{proof}
		This fact may be observed through empirical analysis of the Hurwitz orbits of reflection factorizations of Coxeter elements.  This was done exhaustively using Sage software \cite{sage}.
	\end{proof}
		
\subsection{G4 and G5 Step 4}\label{redux}
	Our goal is to be able to say that if we have a consecutive pair of the form $(t,t)$, we may treat it as though it was the element $t^2$ for the purpose of Hurwitz moves and then split it back into a pair $(t',t')$ at the end.
	
	We begin by defining some structures.
	\begin{definition}\label{marked}
		A \emph{marked element} is a reflection $t$ in a reflection factorization which has been marked, denoted $t^*$.  A \emph{marked factorization}, denoted $T'$, is a reflection factorization which contains a marked element.  The \emph{underlying factorization} $\hat{T}$ is the factorization $T'$ where the marked element $t^*$ is replaced by $t$.
	\end{definition}
	We now explore the functionality of this notation.  To begin with, we would like to perform Hurwitz moves with this marked element, so we need to define Hurwitz moves on marked factorizations.
	\begin{definition}\label{Hurdef}
		The \emph{marked Hurwitz move} $\sigma^*$ is defined as follows for marked reflection factorizations:
		\begin{enumerate}
			\item $(t_1,\,t_2)\xrightarrow{\sigma_1^*}(t_2,\quad t_2^2\cdot t_1\cdot t_2)$
			\item $(t_1,\,t_2^*)\xrightarrow{\sigma_1^*}(t_2^*,\quad t_2^2\cdot t_1\cdot t_2)$
			\item $(t_1^*,\,t_2)\xrightarrow{\sigma_1^*}\big(t_2,\quad (t_2^2\cdot t_1\cdot t_2)^*\big)$
		\end{enumerate}
	\end{definition}
	We now have an interpretation for performing Hurwitz moves with marked elements.  It should be noted that this interpretation is identical to a normal Hurwitz move on the underlying factorization with the additional fact of shifting the position of the mark $^*$.  As such, this definition will preserve all the established facts regarding Hurwitz orbits.\footnote{It can also be shown that this marked move still gives rise to a group action on the braid group, as above.}
	
	As noted, the marked Hurwitz move shifts the position of the marked element.  This raises the question, how relevant is the position of the marked element and can it be changed without altering the elements around it.
	
	Note that the conjugacy class of the marked element never changes.  Thus, there is some limitation to which elements in an orbit can be marked--only those elements which are of a set conjugacy class.  The goal is therefore to show this is the only restriction.
	\begin{proposition}\label{repositionrem}
		Suppose $t_1$ and $t_2$ in the same conjugacy class.  Then $$(t_1,t_2^*)\xrightarrow{(\sigma_1^*)^3}(t_1^*,t_2).$$
	\end{proposition}
	\begin{proof}
		Properties two and three of definition \ref{Hurdef} show that an application of $(\sigma_1^*)^3$ will have the effect of shifting the marked element to the left.  Furthermore, Remark \ref{2facrem} gives that this pair has an orbit of size three, meaning that $(\sigma^*)^{3}$ has no other effect on the underlying factorization.
	\end{proof}
	\begin{proposition}\label{reposprop}
		Two marked factorizations in either G4 or G5 that have the same underlying factorization are in the same Hurwitz orbit if the marked elements are of the same conjugacy class.
	\end{proposition}
	\begin{proof}
		Consider two marked factorizations which have the same underlying factorization.  Assume these marked elements are of the same conjugacy class, say $\calR_2$.  By Proposition \ref{sortingprop} we may select these two factorizations as $(\ldots,t_q^*,\ldots,t_i,t_{i+1},\ldots,t_m)$ and $(\ldots,t_p^*,\ldots,t_i,t_{i+1},\ldots,t_m)$ such that $t_j\in\calR_2$ if and only if $j\in[1,i]$ and $q\geq p$.  By Proposition \ref{repositionrem}, we have the following:
		\begin{align*}
			(\ldots,t_p, t_{p + 1}, \ldots, t_{q - 2}, t_{q - 1}, t_q^*,\ldots)&\xrightarrow{(\sigma_{q-1}^*)^{3}}(\ldots,t_p, t_{p + 1}, \ldots, t_{q - 2}, t_{q - 1}^*, t_q,\ldots)\\
			&\xrightarrow{(\sigma_{q-2}^*)^{3}} (\ldots,t_p, t_{p + 1}, \ldots, t_{q - 2}^*, t_{q - 1}, t_q,\ldots)\\
			&\quad\vdots\\
			&\xrightarrow{(\sigma_{p}^*)^{3}}(\ldots,t_p^*,t_{p+1}, \ldots,t_q,\ldots).
		\end{align*}
		Therefore, the factorizations are in the same Hurwitz orbit, as desired.
	\end{proof}

	With this information, we may define a new object which we will use to create a method for induction.
	\begin{definition}\label{duple}
		Let $T=(\ldots,t^2,t^2,\ldots)$ be a reflection factorization.  Define the \emph{encoded factorization} $T^*$ to be the pair $(\hat{T},\calR_i)$ where $\hat{T}$ is the underlying factorization of a marked factorization $T'$ where $t^*$ replaces $(t^2,t^2)$ and $\calR_i$ is the conjugacy class of $t$. 
	\end{definition}
  
	Our goal is to use this object to be able to treat a length-$m$ reflection factorization of a Coxeter element as one of length $(m-1)$ for the purpose of Hurwitz moves.  Note that $T$ and $T'$ have this relation, however $T'$ uses marked Hurwitz moves.  Thus, we must find some way to reconcile marked and unmarked Hurwitz moves.
	\begin{proposition}\label{hurprop}
		Let $T_1$ be a reflection factorization of the form $(\ldots,t^2,\,t^2,\ldots)$ and $T'_1$ be a marked factorization where $t^*$ replaces the consecutive $(t^2,t^2)$ pair in $T_1$.  If there exists a series of marked Hurwitz moves which may be performed on $T'_1$ to reach a marked reflection factorization $T'_2$ where $t'^*$ replaces a consecutive pair $(t'^2,t'^2)$ in a reflection factorization $T_2$, then there exists a series of Hurwitz moves which can be performed on $T_1$ to reach $T_2$.
	\end{proposition}
	\begin{proof}
		It suffices to prove in the case where $T'_1$ and $T'_2$ are separated by a single marked Hurwitz move. Consider a marked Hurwitz move on $T'_1$.  If this move is performed on a pair of elements not involving the marked element, it will have the same effect when performed on the corresponding pair in $T_1$.  Otherwise consider the following move
		$$
		T'_1=(\ldots,s,\,t^*,\ldots)\xrightarrow{\sigma_i^*}(\ldots,\,t^*,\,t^2\cdot s\cdot t,\ldots):=T'_2.
		$$
		In the left marked factorization, we may replace $t^*$  with $(t^2,t^2)$ to get our factorization $T_1=(\ldots,\,s,\,t^2,\,t^2,\ldots)$.  In the right marked factorization, we may do the same to get the factorization $T_2=(\ldots,\,t^2,\,t^2,\, t^2\cdot s\cdot t,\ldots)$.  Additionally, 
		$$
		T=(\ldots,\,s,\,t^2,\,t^2,\ldots)\xrightarrow{\sigma_i}(\ldots,\,t^2,\, t\cdot s\cdot t^2,\,t^2,\ldots)\xrightarrow{\sigma_{i+1}}(\ldots,\,t^2,\,t^2,\, t^2\cdot s\cdot t,\ldots)=T_2,
		$$
		as desired.
		Now consider the following move:
		$$
		T'_1=(\ldots,t^*,\,s,\ldots)\xrightarrow{\sigma_i^*}(\ldots,s,\,(s^2\cdot t\cdot s)^*,\ldots):=T'_2.
		$$
		In the left marked factorization, we may replace $t^*$  with $(t^2,t^2)$ to get our factorization $T_1=(\ldots,t^2,\,t^2,\,s,\ldots)$.  In the right marked factorization, we may do the same to get the factorization $T_2=(\ldots,\,s,\,s^2\cdot t^2\cdot s,\,s^2\cdot t^2\cdot s,\ldots)$.  Additionally,
		\begin{multline*}
		T_1=(\ldots,t^2,\,t^2,\,s,\ldots)\xrightarrow{\sigma_i}(\ldots,\,t^2,\,s,\, s^2\cdot t^2\cdot s,\ldots) \\ \xrightarrow{\sigma_{i+1}}(\ldots,\,s,\,s^2\cdot t^2\cdot s,\,s^2\cdot t^2\cdot s,\ldots)=T_2,
		\end{multline*}
		as desired.  Since this covers all possible cases for $T'_1$ and $T'_2$ being separated by a single marked Hurwitz move, this concludes the proof.  
	\end{proof}
	\begin{lemma}\label{reduxlem}
		Let $T_1$ and $T_2$ be two reflection factorizations that contain a consecutive $(t,t)$ pair and have encoded factorizations $T_1^*=(\hat{T_1},\calR_{i})$ and $T_2^*=(\hat{T_2},\calR_j)$.  Then $T_1$ and $T_2$ are in the same Hurwitz orbit if $\hat{T_1}$ and $\hat{T_2}$ are in the same Hurwitz orbit and $i=j$.
	\end{lemma}
	\begin{proof}
		Assume $\hat{T_1}$ and $\hat{T_2}$ are in the same Hurwitz orbit and $i=j$. Consider their associated marked factorizations $T_1'$ and $T_2'$.  Since their underlying factorizations are in the same orbit we may perform some series of marked Hurwitz moves on $T'_1$ so it is identical $T'_2$ up to which element is marked.  By Proposition \ref{reposprop}, we may reposition this marked element until the resulting factorization is equal to $T_2'$.  By Proposition \ref{hurprop} there exists a braid action which can be performed on $T_1$ to reach $T_2$.  Thus, $T_1$ and $T_2$ are in the same Hurwitz orbit, as desired.   
	\end{proof}
\subsection{Proof of Theorem}	
	We now come to a proof of theorem \ref{transthm} for G4, restated here for convenience, which follows from a simple induction using Lemma \ref{reduxlem} on a base case founded by Bessis' Theorem.
  
	\begin{recap}
		Let $W$ be one of the reflection groups G4 or G5 and let $c\in W$ be a Coxeter element. Two reflection factorizations $(t_1,t_2,\ldots,t_k)$ and $(t_1',t_2',\ldots, t_k')$ of $c$ are in the same Hurwitz orbit if and only if they have the same multiset of conjugacy classes.
	\end{recap}
	\begin{proof}[Proof in the case of G4.]
		\noindent``$\Rightarrow$" 
		
		Follows the basic facts of Hurwitz moves as stated in Remark \ref{multisetrem}.
		
		\noindent``$\Leftarrow$"  
		
		We conduct a proof by induction on the length $m$ of the factorizations.  Fix a Coxeter element $c$.  The base case $m=2$ is given by Theorem \ref{thm:Bessis} as these are the shortest reflection factorizations of Coxeter elements.  Assume the property holds for factorizations of length $k$.  Consider those of length $(k+1)$.  Take two reflection factorizations of $c$ with the same multiset of conjugacy classes.  Lemma \ref{pairlem} states that both of these have, in their Hurwitz orbit, a factorization with a consecutive pair of equal reflections.  Observe from the proofs of Propositions \ref{3facprop}, \ref{3facprop2}, and \ref{4facprop} that the conjugacy class of a consecutive $(t,t)$ pair which appears in the factorization's orbit depends only on the starting configuration of conjugacy classes and strategy used.  Since our two chosen factorizations have the same multiset of conjugacy classes, we may ensure that the selected consecutive $(t,t)$ pairs in each uses the same starting configuration of conjugacy classes and the same strategy from one of these proofs.  Call the factorizations in which these pairs appear $T_1$ and $T_2$.  From these, we may construct encoded factorizations $T_1^*$ and $T_2^*$ with the equal $\calR_i$ elements.  By the inductive hypothesis, $\hat{T_1}$ and $\hat{T_2}$ are in the same Hurwitz orbit.  Therefore, by Lemma \ref{reduxlem}, so are $T_1$ and $T_2$, as desired.
	\end{proof}

    The proof of Theorem \ref{transthm} for G5 is very similar to that of G4, as intended.
	\begin{proof}[Proof in the case of G5.]
		The proof for G5 is very similar.  It differs in the following manner.  As stated above, it has two additional base cases beyond $m=2$ proven by Bessis, namely $m=3,m=4$.  Both are given by Proposition \ref{g5induction}.  Rather than appealing to Lemma \ref{pairlem} for the existence of consecutive pairs, we appeal to Lemma \ref{g5pairlem}.  We must then check that the proofs of Propositions \ref{4facprop}, \ref{g5threetwo}, and \ref{g5twotwoone} ensure that we may pick the conjugacy class of the desired pair with the proper set-up, which they do.  From here, the proof is identical.
\end{proof}

\section{Further remarks and open problems}\label{P.S.}

\subsection{Other groups with generators of order $3$}~ 
	
	Lemma \ref{reduxlem} makes use of the concept of treating consecutive $(t,t)$ pairs as a single reflection $t^2$.  This concept is dependent on the fact that $t^2=t^{-1}$ for any reflection $t$, a fact which is logically equivalent to all reflections in a complex reflection group being of order $3$.  As such, it is conceivable that, with some alterations, a strategy like the one above may be used to prove the conjecture for these other groups.
	\begin{remark}\label{othergroups}
		Of the irreducible Shepherd-Todd groups, G4, G5, G20, G25, and G32 all have exclusively reflections of order 3.
	\end{remark}
	We performed some additional research on these groups, the majority with G25 due to properties which bear certain parallels to G4 (see Section \ref{g25sec} below).  Though we did not reach a proof of the conjecture for G20, G25, or G32, we did draw some conclusions about the method.  Our method of induction was chosen because it is relatively "opposite" to the following remark.
	
	\begin{remark}
		Let $W$ be a reflection group with exclusively generators of order three.  Let $c\in W$ be a Coxeter element.  Let $T=(t_1,\,t_2,\ldots,\,t_m)$ be a length-$m$ reflection factorization of $c$.  The factorization $T^+=(t_1^2,\,t_1^2,\,t_2,\ldots,\,t_m)$ is a length-$(m+1)$ reflection factorization of $c$.
	\end{remark}
	
	We utilize this result combined with the following conjecture to a priori find the possible multisets of conjugacy classes for the reflection factorizations of Coxeter elements of any given length.
	
	\begin{conjecture}[Corollary to Conj. \ref{conj}]\label{genconj}
		Let $W$ be a reflection group with exclusively generators of order three.  Let $c\in W$ be a Coxeter element.  Let $T$ be a longer reflection factorization of $c$.  This factorization has, in its Hurwitz orbit, a reflection with a consecutive $(t,t)$ pair.
	\end{conjecture}
	
	Conjecture \ref{genconj} is the most general versions of our Lemmas \ref{pairlem} and \ref{g5pairlem}.  Proving it in specific cases, along with some version of Lemma \ref{reduxlem} yields Conjecture \ref{conj}.  This fact gives rise to the following.
	
	\begin{conjecture}
		Let $W$ be a reflection group with exclusively generators of order three.  The following are logically equivalent.
		\begin{enumerate}
			\item Conjecture \ref{conj}
			\item Conjecture \ref{genconj}
		\end{enumerate}	
	\end{conjecture} 
	
	Another group outside the complex reflection groups generated exclusively by elements of order three is the alternating group $A_n$ of even permutations generated by $3$-cycles.  The Hurwitz action on these groups has been studied by M{\"u}hle and Nadeau who have proven, among other things, that if $n$ is odd then the Hurwitz action is transitive on the shortest factorizations of an $n$-cycle as a product of $3$-cycles \cite[Thm 1.2]{MN}.  This theorem's resemblance to Bessis' Theorem \ref{thm:Bessis} gives rise to the obvious question. 
	\begin{question}
		Can the methods presented here be used in a similar manner to prove a conjecture similar to Conjecture \ref{conj} for the alternating groups $A_n$ where $n$ is odd?
	\end{question}
	
\subsection{Remarks about G25}\label{g25sec}
	
	The following are some remarks about G25.  The complex reflection group G25 is of rank three and is minimally generated by three reflections $A$, $B$, and $C$ which have the relations
	\begin{gather*}
		ABA=BAB\\
		CBC=BCB\\
		AC=CA\\
		A^3=B^3=C^3=I.
	\end{gather*}
	The 24 reflections are separated into two conjugacy classes of size twelve.  The 24 Coxeter elements are divided similarly.  As with G4, this separation is by eigenvalues.  Unlike G4, there are some additional structures within the conjugacy classes of reflections.  Within each conjugacy class, there are four triples of reflections which commute with each other.  It should be noted that the the inverses of the elements within a triple lie in a single triple in the other class, and that all six of these elements commute.
	\begin{remark}
		Consider reflections $t$ and $t^2$.  The conjugacy class of $t^2$ contains the squares of all elements of the conjugacy class of $t$.
	\end{remark}	
	Remark \ref{2facrem} holds with the updated constraint that the elements not be from the same triple.  Remark \ref{neqrem} becomes the following.  
	\begin{remark}
		A factorization which consists of exactly one pair of elements from different conjugacy classes and not inverse triples has orbit of size four in which two elements from each conjugacy class appear.  None of these elements are in the same triples or inverse triples.
	\end{remark}
	For elements within the same triple or from inverse triples, we recall that the Hurwitz orbit of a factorization which consists of a pair of reflections which commute has size two.
	
	Thus, with a few complications, G25 is very similar to G4.  The convenient part about G4, however, was that the size of its conjugacy classes was relatively small, which meant these orbits contained greater portions of the classes, making it easier to find $(t,t)$ pairs, or at least to prove their existence.
	\begin{proposition}\label{g25l6}
		Conjecture \ref{conj} holds for G25 up to length-six reflection factorizations.
	\end{proposition}
	\begin{proof}
		Bessis' Theorem gives that this is true for length-3 reflection factorizations of Coxeter elements.  This fact has been observed through exhaustive empirical analysis of the Hurwitz orbits of length-4, length-5, and length-6 reflection factorizations of Coxeter elements.
	\end{proof}
	\begin{proposition}\label{g25l13}
		Each reflection factorization of length at least 13 in G25 has in its Hurwitz orbit a reflection factorization with a desired pair.  Each reflection factorization in G25 with at least nine reflections of the same conjugacy class has in its Hurwitz orbit a reflection factorization with a desired pair.    
	\end{proposition}
	\begin{proof}[Proof sketch]
		 Consider the first condition.  Define the reflection factorization $T$ with the desired property.  We may assume, without loss of generality, that this contains elements of $\calR_1$.  Order $T$ such that $t_1$ is in $\calR_1$ and alternate between classes until for as long as possible.  This can be done in a manner such that no consecutive elements commute.  Consider the orbit of $(t_1,t_2)$.  An additional element of $\calR_1$ appears.  Note this element.  If it does not appear in the rest of $T$, repeat this process for $(t_3,t_4)$.  If this new element of $\calR_1$ does not appear in the rest of $T$ and is not the first noted element, continue to $(t_5,t_6)$.  Iterate this process a further four times.  This process has generated 6 new elements of $\calR_1$.  Since $T$ began with at least 7 and there are only 12 elements of $\calR_1$, there must have been a repeated element at some step by the pigeonhole principle.
		 
		 The proof of the second condition is very similar from the starting position of all nine elements of the same conjugacy class arranged consecutively.
	\end{proof}  
   
	The combination of these three propositions leaves a small gap for which we have not yet found a way to prove that a reduction is possible.  The length-seven reflection factorizations a probably within the realm of brute force calculation by a supercomputer.  Additionally, of the length-twelve factorizations, the only one not covered by the above proposition is the one with multiset of conjugacy classes $\{1^6,2^6\}$.
	
	Another small difficulty is with Proposition \ref{repositionrem}.  This depends on the orbit of two elements of the same conjugacy class to be odd which is not always the case in G25 given elements from the same triple.  The fix for this is presented as follows.
	\begin{proposition}\label{G25reposprop}
		Two marked factorizations of a Coxeter element in G25 that have the same underlying factorization are in the same Hurwitz orbit if the marked elements are of the same conjugacy class.
	\end{proposition}
	\begin{proof}
		Consider two marked factorizations which have the same underlying factorization.  Assume these marked elements are of the same conjugacy class, say $\calR_2$.  By Proposition \ref{sortingprop} we may select these two factorizations as $(t_1,\ldots,t_q^*,\ldots,t_i,t_{i+1},\ldots,t_m)$ and $(t_1,\ldots,t_p^*,\ldots,t_i,t_{i+1},\ldots,t_m)$ such that $t_j\in\calR_2$ if and only if $j\in[1,i]$ and $q\geq p$.  Suppose $t_p$ and $t_q$ are not in the same triple.  Then we may perform the following series of Hurwitz moves.
		\begin{align*}
		(t_1,\ldots,t_p,t_{p+1},\ldots,t_{q-1},t_q^*,\ldots)&\xrightarrow{\sigma_{q-1}^*}(t_1,\ldots,t_p,t_{p+1},\ldots,t_q^*,t_{q-1}',\ldots)\\
		&\quad\vdots\\
		&\xrightarrow{\sigma_{p+1}^*}(t_1,\ldots,t_p,t_q^*,\ldots)\\
		&\xrightarrow{(\sigma^*_p)^3}(t_1,\ldots,t_p^*,t_q,\ldots)\\
		&\xrightarrow{(\sigma^*_{p+1})^{-1}}\ldots\\
		&\quad \vdots\\
		&\xrightarrow{(\sigma_{q-1}^*)^{-1} }(t_1,\ldots,t_p^*,t_{p+1},\ldots,t_{q-1},t_q,\ldots)
		\end{align*}
		Suppose that $t_p$ and $t_q$ are in the same triple.  There are two possible methods of procedure.  The first is to find another reflection in the same conjugacy class but not in the same triple and transfer the mark to it using a similar process as above and then repeating this process to transfer the mark to $t_p$.
		
		The second method is to find a reflection of the other conjugacy class but not in the same triple and perform a single Hurwitz move on this consecutive pair.  Marked element will now be of the same conjugacy class but a different triple and the above procedure may be performed.
		
		Suppose neither method is possible.  By the pigeonhole principle, this means that there can be at most three reflections from each conjugacy class, meaning the total length of the factorization could be at most six.  However, Proposition \ref{g25l6} shows that this case is impossible.
	\end{proof}
	Thus, Theorem \ref{transthm} will apply if it can be shown that Conjecture \ref{genconj} applies.
	
\subsection{Counting factorizations}
	We gathered some additional data during the course of our research, but which was not relevant above, to do with counting the number and size of the Hurwitz orbits of reflection factorizations of Coxeter elements in G4 specifically.  These discoveries are presented here for the sake of completeness.
	
	Currently, there is no general formula for calculating the size of a Hurwitz orbit of a given reflection factorization.  There are, however, multiple methods for counting the total number of reflection factorizations of a given length within a group.  For instance, Bessis \cite[Prop. 7.6]{Bess15} proved that the orbit of a shortest reflection factorization has size $n!h^n/|W|$. More generally, Chapuy and Stump \cite{CS}, proved that the number of reflection factorizations of a given length of a Coxeter element in a group $W$ is given by the exponential generating function
	$$
	\frac{1}{|W|}\left(e^{q|\calR|/n}-e^{-q|\mathcal{A}|/n}\right)^n.
	$$
	A more refined result for the complex reflection groups was developed by Del Mas, Hameister, and Reiner \cite{DHR}.
 
	In the course of this research on G4, certain patterns were observed that gave rise to a more complete counting for the size of the orbit of a given factorization.  These findings are presented here.
	
	Since we are dealing exclusively with G4, we may specialize the formula of Chapuy and Stump to get the following remark.
	\begin{remark}\label{countingrem}
		The number of length-$m$ reflection factorizations of a Coxeter element in G4 is
		$$
		\frac{1}{24}\bigg(8^m-2^{m+1}+(-4)^m \bigg).
		$$ 
	\end{remark} 
	This follows from the following calculation using Chapuy and Stump's formula \cite{CS} using $|W|=24$, $|\calR|=8$, and $|\mathcal{A}|=4$.
	\begin{align*}
	\frac{1}{24}\big(e^{4q}-e^{-2q} \big)^2&=\frac{1}{24}(e^{8q}-2e^{2q}+e^{-4q})\\
	&=\frac{1}{24}\bigg(\sum_{i\geq0}\frac{8^iq^i}{i!}-2\sum_{i\geq0}\frac{2^iq^i}{i!}+\sum_{i\geq0}\frac{(-4)^iq^i}{i!}\bigg)\\
	&=\sum_{i\geq0}\bigg(\frac{8^i-2^{i+1}+(-4)^i}{24}\bigg)\frac{q^i}{i!}
	\end{align*}
	which has the desired coefficient.
	
	Let us consider the two conjugacy classes of reflections $\calR_1,\calR_2$ and the classes of Coxeter elements $\calC_1,\calC_2$.  Assign to each of the elements in these groups a value equal to the index of their class.
	\begin{proposition}\label{modsprop}
	 	The sum of the values of the reflections in a reflection factorization of a Coxeter element is congruent mod 3 to the value of the Coxeter element.
	\end{proposition}
	\begin{proof}
	 	Note that the determinants of elements from $\calR_1$ and $\calC_1$ are $\omega$ and those from $\calR_2$ and $\calC_2$ are $\omega^2$, meaning the value of an element in $\calR$ or $\calC$ corresponds to the exponent of $\omega$ in its determinant.  Also note that multiplication of $\omega$'s is equivalent to mod 3 addition of their exponents.  Recall that the product of the determinants of square matrices of the same size is the determinant of the product of those matrices.  The statement follows.
	\end{proof}
	This proposition allows us to define a method of thinking about factorizations in terms of things that are well-studied in the fields of combinatorics, namely multisets of integers, and number theory, namely modular addition.
	\begin{proposition}\label{numprop}
	 	For $m\geq2$, the number of Hurwitz orbits of length-$m$ reflection factorizations of a Coxeter element is $\lceil m/3 \rceil$.
	\end{proposition}
	\begin{proof}
	 	We prove that the number of multisets $\{1^a,2^b\}$ with $a+b=m$ and sum congruent to 1 mod 3 is $\lceil m/3\rceil$.  The proof for 2 is very similar and Proposition \ref{modsprop} shows that 3 is impossible for Coxeter elements.  The minimum value for the sum of the numbers in the multiset is $m$, which is reached if all $m$ values are 1's.  Every value up to and including $2m$ can be reached by replacing a 1 with a 2.  Thus, we need only count the number of values in the range $[m,\,2m]$ that are congruent to 1 mod 3.  If $\,m\equiv0\pmod3$, then the range has $m/3$ values congruent to 1, namely $m+1+3(0),\,m+1+3(1),\ldots,\,m+1+3(m/3-1)$.  If $\,m\equiv1\pmod3$ then the values congruent to 1 are $m+3(0),\,m+3(1),\ldots,\,m+3(\lfloor m/3\rfloor)$.  If $m\equiv2\pmod3$ the values are very similar.  In each case, the number of values congruent to 1 mod 3 was $\lceil m/3\rceil$ as desired.
	\end{proof}
	This use of multisets also informs us as to the relative size of these orbits through the use of binomials.  Data for $m\leq7$ suggests the following conjecture.
	\begin{conjecture}\label{ratiorem}
		 Let $T_1$ and $T_2$ be length-$m$ reflection factorizations of a Coxeter element $c$ with multisets of conjugacy class values, $\{1^a, 2^b \},\{1^{a'},2^{b'} \}$.  The ratios of the sizes of their orbits is equal to the ratio $\binom{a}{a+b}/\binom{a'}{a'+b'}$.
	\end{conjecture}
	The proof should be similar to those by Chapuy-Stump and delMas-Hameister-Reiner, but is beyond the scope of this paper.
	
	We may combine these results to get a complete counting for the size and number of Hurwitz orbits of length $n$ reflection factorizations of a Coxeter element in G4.

	\section*{Acknowledgments}
	The author would like thank Svetlana Rudenko for encouraging him to start this project.  Additionally, he would like to thank Elizabeth Drellich and Victor Reiner for contributions made in conversations.  In particular, he would like to thank Dr. Drellich for providing the code for Figure \ref{G4im}.  Finally, he would like to thank Joel Brewster Lewis for major contributions of mentorship and guidance.

	The author acknowledges the sponsorship of the GW Data MASTER Program, supported by the National Science Foundation under grant DMS-1406984.
	
\end{document}